\documentclass[12pt,leqno]{amsart}
\usepackage{verbatim,amssymb,amsfonts,latexsym,amsmath,amsthm,tikz}
\usetikzlibrary{arrows}

\newtheorem{theorem}{Theorem}[section]
\newtheorem{lemma}[theorem]{Lemma}
\newtheorem{corollary}[theorem]{Corollary}
\newtheorem{proposition}[theorem]{Proposition}
\newtheorem{question}[theorem]{Question}
\theoremstyle{definition}

\theoremstyle{remark}
\newtheorem{remark}[theorem]{Remark}

\newtheorem*{sub-claim}{sub-claim}

\newcommand{\F}{\mathcal{F}}
\newcommand{\G}{\mathcal{G}}
\newcommand{\N}{\mathbb{N}}
\newcommand{\explicitSet}[1]{\left\lbrace #1 \right\rbrace}
\newcommand{\brackets}[1]{\left\langle #1 \right\rangle}
\newcommand{\set}[2]{\explicitSet{#1 \colon #2}}
\newcommand{\seq}[2]{\brackets{#1 \colon #2}}

\newcommand{\0}{\emptyset}
\renewcommand{\a}{\alpha}
\renewcommand{\b}{\beta}

\newcommand{\card}[1]{|#1|}
\newcommand{\sub}{\subseteq}
\newcommand{\rest}{\!\restriction\!}
\newcommand{\w}{\omega}
\newcommand{\continuum}{\mathfrak{c}}

\newcommand{\homeo}{\approx}

\newcommand{\closure}[1]{\overline{#1}}

\newcommand{\s}{\sigma}
\renewcommand{\t}{\tau}
\newcommand{\B}{\mathcal{B}}
\newcommand{\U}{\mathcal{U}}

\newcommand{\thick}{\Theta}

\newcommand{\plim}{p\mbox{-}\!\lim_{n \in \N}}
\newcommand{\qlim}{q\mbox{-}\!\lim_{n \in \N}}

\newcommand{\iso}{\cong}

\renewcommand{\tt}{\mathfrak{t}_\thick}
\newcommand{\tower}{\mathfrak{t}}
\newcommand{\pin}{\mathfrak{p}}
\renewcommand{\P}{\mathbb{P}}
\renewcommand{\k}{\kappa}
\newcommand{\Inf}{\mathcal P (\w) / \text{\textit{fin}}}

\begin{document}

\title{$P$-sets and minimal right ideals in $\N^*$}
\author{W. R. Brian}
\address {
William R. Brian\\
Department of Mathematics\\
Tulane University\\
6823 St. Charles Ave.\\
New Orleans, LA 70118}
\email{wbrian.math@gmail.com}

\begin{abstract}
Recall that a $P$-set is a closed set $X$ such that the intersection of countably many neighborhoods of $X$ is again a neighborhood of $X$. We show that if $\tower = \continuum$ then there is a minimal right ideal of $(\b\N,+)$ that is also a $P$-set. We also show that the existence of such $P$-sets implies the existence of $P$-points; in particular, it is consistent with ZFC that no minimal right ideal is a $P$-set. As an application of these results, we prove that it is both consistent with and independent of ZFC that the shift map is (up to isomorphism) the unique chain transitive autohomeomorphism of $\N^*$.
\end{abstract}

\maketitle

\section{Introduction}

The dynamical and algebraic structure of $\b\N$ and $\N^*$ has played a major part in modern combinatorics and algebra. The minimal right ideals of $\b\N$ -- or, equivalently, its minimal dynamical subsystems -- have a special place in this theory. In this paper we study how these ideals are embedded in $\N^*$.

Our main result is that it is both consistent with and independent of ZFC that some minimal right ideal is also a $P$-set.

In Section~\ref{sec:thickfilters}, we interpret the minimal right ideals of $\N^*$ as the ultrafilters on a certain partial order. This provides a tool for studying how these ideals are embedded in $\N^*$, and some of the techniques for analyzing the topology of $\N^*$ (the set of all ultrafilters on a different partial order, namely $\Inf$) carry over. In this way we are able to prove that if $\tower = \continuum$ then some minimal right ideal is also a $P$-set. Several corollaries of this result will be given, for example that there is an idempotent ultrafilter that is both minimal and maximal.

In Section~\ref{sec:independence}, we show that if $I$ is a minimal right ideal and a $P$-set, then there is a finite-to-one $f: \N \to \N$ such that $\b f(I)$ is a $P$-point. It follows that it is consistent to have no minimal right ideal be a $P$-set. We also show in this section that, under ZFC, there is a nowhere dense closed right ideal that is a $P$-set.

In Section~\ref{sec:shiftmap}, we give an application of these results by obtaining another consistency/independence theorem. The shift map, the canonical dynamical structure on $\N^*$ obtained by extending the map $n \mapsto n+1$ to $\b\N$, is shown to be a chain transitive map. We will prove that it is consistent with and independent of ZFC that the shift map is the only chain transitive autohomeomorphism of $\N^*$. The proof of independence will use the fact that some nowhere dense right ideal is a $P$-set, together with Parovi\v{c}enko's topological characterization of $\N^*$ under CH.

\section{Notation and preliminaries}

\subsection*{Ultrafilters}
Let $(\P,\leq)$ be a partial order. A \textbf{filter} on $\P$ is a subset $\F$ of $\P$ satisfying
\begin{itemize}
\item (nontriviality) $\0 \neq \F \neq \P$.
\item (upwards heredity) if $a \in \F$ and $a \leq b$ then $b \in \F$.
\item (finite intersection property) if $a,b \in \F$ then there is some $c \in \F$ with $c \leq a$ and $c \leq b$.
\end{itemize}
An \textbf{ultrafilter} on $\P$ is a maximal filter: a filter $\F$ such that no set properly containing $\F$ is also a filter. Ultrafilters on the partial order $(\mathcal P(\N),\sub)$ are perhaps most familiar, and here we will call these simply \textbf{ultrafilters on $\N$} or, context permitting, ultrafilters. A \textbf{free} (ultra)filter on $\N$ is one that contains every cofinite set.

If $\F$ satisfies the nontriviality and finite intersection properties, then $\F$ is called a \textbf{filter base}, and $\set{b \in P}{a \leq b \text{ for some } a \in \F}$ is a filter, namely the filter \textbf{generated by} $\F$. A \textbf{chain} in $\P$ is a totally ordered subset of $\P$. Note that every chain is a filter base.

If $\F$ satisfies the nontriviality property and is upwards hereditary, then $\F$ is an \textbf{up-set} of $\P$ (the up-sets of $(\N,\sub)$ are often referred to as \emph{Furstenberg families} in the dynamical systems literature; the up-sets of $\Inf$ are often referred to as \emph{semifilters} in the set theory literature).

Following convention, if $A,B \sub \N$ we write $A \sub^* B$ whenever $A \setminus B$ is finite and we write $A =^* B$ whenever $A \sub^* B$ and $B \sub^* A$. $\Inf$ is the set of $=^*$-equivalence classes of infinite subsets of $\N$, naturally ordered by $\sub^*$. $\Inf$ is a well-studied Boolean algebra, and the ultrafilters on the partial order $\Inf - \{0\}$ are the well-known space $\N^*$. From now on, we will abuse notation slightly and write $\Inf$ to mean the partial order $\Inf - \{0\}$.

Given $(\P,\leq)$ and $A \sub \P$, $p \in \P$ is a \textbf{(lower) bound} for $A$ provided $p \leq a$ for every $a \in A$. $A$ is \textbf{unbounded} if it has no bound. Note: ``pseudo-intersection'' is often used synonymously with ``lower bound'' when working with $\Inf$.

\subsection*{Some small cardinals} Much has been written about the ``small cardinal'' numbers, of which there are many. Here we define four, three old and one new (the introduction of a new one can be forgiven, since we will show in Theorem~\ref{thm:cardinals} it is really just $\tower$ in disguise). Our definitions are not quite standard. This is done intentionally, to emphasize that our main theorem is really a theorem about a certain partial order $\thick$. In fact, the proof of our main theorem can be viewed as the translation to $\thick$ of a well-known proof done in $\Inf$.

Recall that $A \sub \N$ is \textbf{thick} if it contains arbitrarily long intervals. Equivalently, $A$ fails to be thick if and only if there is some $k$ such that $\bigcap_{n \leq k}(A - n)$ is empty (or if and only if $\bigcap_{n \leq k}(A+n)$ is finite). The thick sets form an up-set of $\mathcal \P(\N)$. If $A =^* B$, then $A$ is thick if and only if $B$ is thick, so the (equivalence classes of) thick sets also form an up-set of $\Inf$. Let $\thick$ denote the family of thick sets modulo $=^*$, partially ordered by $\sub^*$. We refer to the (ultra)filters on $(\thick,\sub^*)$ as $\thick$\textbf{-(ultra)filters}.

Recall that $\N^\N$ denotes the set of all functions $\N \to \N$, and if $f,g \in \N^\N$, then $f <^* g$ if and only if $\set{n \in \N}{f(n) \not < g(n)}$ is finite.

We now define the small cardinals we will need:
\begin{itemize}
\item $\tt$ is the smallest size of an unbounded chain in $\thick$.
\item $\tower$ is the smallest size of an unbounded chain in $\Inf$.
\item $\pin$ is the smallest size of an unbounded filter base in $\Inf$.
\item $\mathfrak{b}$ is the smallest size of an unbounded subset of $(\N^\N,>^*)$.
\end{itemize}

Defined this way, it is easy to see that $\tt$ is simply the analogue of $\tower$ in $\thick$. While the definition of $\tt$ is ours, the last three cardinals here are standard, and we leave it to the reader to verify that these definitions are equivalent to the usual ones given in \cite{EvD} or \cite{AB2}.

Observe that the partial order $(\N^\N,>^*)$ has the property of being a filter. Using this, $\mathfrak{b}$ can also be characterized as the smallest size of an unbounded chain in $(\N^\N,>^*)$ or an unbounded filter base in $(\N^\N,>^*)$ (see also \cite{EvD}, Theorem 3.3). Thus $\mathfrak{b}$ is also just the analogue of $\tower$ (or of $\pin$) in a different partial order.

\subsection*{$\b\N$ and Stone duality}
The set of all ultrafilters on $\N$ is denoted $\b\N$. As usual, we identify the principal ultrafilter $\set{A \sub \N}{n \in A}$ with $n$, so that $\N \sub \b\N$. For each $A \sub \N$, $\closure{A} = \set{p \in \b\N}{A \in p}$, and taking the sets of this form as the basis for a topology, $\b\N$ is the Stone-\v{C}ech compactification of $\N$. This is the unique compactification of $\N$ with the extension property: for any compact Hausdorff space $K$ and any map $f: \N \to K$, there is a (unique) continuous extension $\b f: \b\N \to K$. We write $\N^* = \b\N \setminus \N$ for the set of all free ultrafilters on $\N$, and, for every $A \sub \N$, $A^* = \closure{A} \cap \N^*$.

If $\F$ is a free (ultra)filter on $\N$, then $\F$ restricts to an (ultra)filter on $\Inf$. Conversely, an (ultra)filter on $\Inf$ gives rise to a unique free (ultra)filter. In what follows we will freely conflate these two notions. Topologically this makes no difference, because $A =^* B$ if and only if $A^* = B^*$.

If $\F$ is a free filter, then $\F^* = \bigcap_{X \in \F}A^*$ is a nonempty closed subset of $\N^*$. Conversely, for every closed $X \sub \N^*$ there is a unique free filter $\F$ such that $\F^* = X$, namely $\F = \set{A \sub \N}{X \sub A^*}$. This correspondence, part of what is called Stone duality, will be used frequently (often implicitly) in what follows.

A closed subset $X$ of $\N^*$ is a $P$\textbf{-set} if and only if, for any countable collection $\mathcal A$ of open sets containing $X$, $X$ is in the interior of $\bigcap \mathcal A$. If $X = \{x\}$ is a $P$-set, then $x$ is called a $P$\textbf{-point}. If $\F$ is a filter such that $\F^*$ is a $P$-set, we say that $\F$ is a $P$\textbf{-filter}. Using Stone duality, $\F$ is a $P$-filter if and only if every countable subset of $\F$ has a lower bound in $\F$.

\subsection*{Dynamics and algebra}
A \textbf{dynamical system} is a pair $(X,f)$, where $X$ is a compact Hausdorff space and $f: X \to X$ is continuous. For our purposes, $X$ will usually be $\b\N$ or $\N^*$.

For each $p \in \N^*$, define $\s(p)$ to be the unique ultrafilter generated by $\set{A+1}{A \in p}$. This is the \textbf{shift map} on $\b\N$, and whenever we speak of $\b\N$ as a dynamical system it is understood that we are talking about the shift map. The shift map is the unique continuous extension to $\b\N$ of the map on $\N$ given by $n \mapsto n+1$.

If $\seq{x_n}{n \in \N}$ is a countable sequence of points in any compact Hausdorff space and if $p \in \b\N$, we define $\plim x_n$ to be the unique element of $\bigcap_{A \in p}\closure{\set{x_n}{n \in A}}$. Then, for every $p,q \in \b\N$, define $p+q = \qlim \s^n(p)$. This is the standard algebraic structure on $\b\N$ (or one of the two -- multiplication also has a natural extension to $\b\N$).

Some authors (see, e.g., \cite{Brg}), define $+$ differently on $\b\N$:
$$p+q = \set{A \sub \N}{\set{n \in \N}{(A - n) \in p} \in q}.$$
This definition is equivalent to the one we give here (the equivalence is well known; see \cite{Bls} for some discussion). Warning: some authors, including some that we cite below, exchange the roles of $p$ and $q$ in this definition.

The map $q \mapsto p+q$ (with $p$ fixed) is continuous, but the map $p \mapsto p+q$ (with $q$ fixed) is not. This makes $\b\N$ a left-topological semigroup. Its \textbf{right ideals} are those $I \sub \b\N$ such that $I + \b\N \sub I$. A \textbf{minimal} right ideal is a right ideal that does not properly contain any other right ideal.

For any dynamical system $(X,f)$, $Y$ is a \textbf{subsystem} of $X$ provided that $Y$ is closed and $f(Y) \sub Y$. A subsystem is \textbf{minimal} if it does not properly contain any other subsystems.

A simple application of Zorn's Lemma shows that every subsystem of $\b\N$ contains a minimal subsystem, and every right ideal contains a minimal right ideal. In $\b\N$, these are related by the following:

\begin{proposition}\label{prop:ideals}
$X$ is a subsystem of $(\b\N,\s)$ if and only if it is a closed right ideal of $(\b\N,+)$. $X$ is a minimal subsystem if and only if $X$ is a minimal right ideal.
\end{proposition}
\begin{proof}
See \cite{Brg}, Theorem 2.1.
\end{proof}

\section{$\thick$-ultrafilters}\label{sec:thickfilters}

The following two lemmas make the connection between $\thick$-ultrafilters and minimal right ideals. These lemmas are hinted at in \cite{BHM}, Theorem 2.9(c).

\begin{lemma}\label{lem:thickfilters}
Let $\F$ be a free filter. $\s(\F^*) \sub \F^*$ if and only if for every $A \in \F$, $A-1 \in \F$. $\F^* = \s(\F^*)$ if and only if, in addition, $A+1 \in \F$ for every $A \in \F$.
\end{lemma}
\begin{proof}
The ``only if'' direction of both claims follow easily from Stone duality and the continuity of $\s^{-1}$ and $\s$:
$$\s^{-1}(\F^*) = \s^{-1}\left(\bigcap_{A \in \F}A^*\right) = \bigcap_{A \in \F}\s^{-1}(A^*) = \bigcap_{A \in \F}(A-1)^*.$$
$$\s(\F^*) = \s\left(\bigcap_{A \in \F}A^*\right) = \bigcap_{A \in \F}\s(A^*) = \bigcap_{A \in \F}(A+1)^*,$$
We can then use these equalities to prove the ``if'' direction of these claims: if $A-1 \in \F$ for every $A \in \F$ then
$$\s^{-1}(\F^*) =\bigcap_{A \in \F}(A-1)^* \supseteq \bigcap_{A \in \F}A^* = \F^*,$$
which implies $\s(\F^*) \sub \F^*$. If also $A-1 \in \F$ for every $A \in \F$ then
$$\s^(\F^*) = \bigcap_{A \in \F}(A+1)^* \supseteq \bigcap_{A \in \F}A^* = \F^*.$$
In the latter case, we clearly have $\F^* = \s(\F^*)$.
\end{proof}

\begin{lemma}\label{lem:thickultrafilters}
$\F^*$ is a minimal right ideal if and only if $\F$ is a $\thick$-ultrafilter.
\end{lemma}
\begin{proof}
If $\F^*$ is a closed right ideal, it follows directly from Theorem 2.9(c) in \cite{BHM} that $\F$ is a $\thick$-filter.

Let $\F$ be a $\thick$-ultrafilter and $A, B \in \F$. $A \cap B$ is thick, and it is clear that if there is an interval of length $k$ in $A \cap B$ then there is an interval of length $k-1$ in $(A+1) \cap B$ and in $(A-1) \cap B$. Thus $(A+1) \cap B$ and $(A-1) \cap B$ are also thick. By the maximality of $\thick$-ultrafilters, $A+1 \in \F$ and $A-1 \in \F$. By Lemma~\ref{lem:thickfilters}, $\F^* = \s(\F^*)$. Thus, by Proposition~\ref{prop:ideals}, if $\F$ is a $\thick$-ultrafilter then $\F^*$ is a closed right ideal.

An easy application of Zorn's Lemma shows that every $\thick$-filter can be extended to a $\thick$-ultrafilter. Together with the preceding two paragraphs and the fact that every minimal right ideal is closed, this proves the lemma.
\end{proof}

By Lemma~\ref{lem:thickultrafilters}, some minimal right ideal is a $P$-set if and only if some $\thick$-ultrafilter is also a $P$-filter. By Stone duality, a $P$-point exists if and only if some ultrafilter on $\Inf$ is also a $P$-filter. This is what we mean when we say that our main result is a translation to $\thick$ of a theorem in $\Inf$.

The first proof of the existence of $P$-points under MA, given by Booth in \cite{Bth} as a modification of Rudin's proof from CH in \cite{Rud}, uses $\mathfrak{t} = \continuum$ to build a $\continuum$-length chain in $\Inf$ such that the filter generated by this chain is both an ultrafilter and a $P$-filter.

To prove that it is possible to have $P$-sets that are minimal right ideals, we will do the same thing as Booth, only replacing $\Inf$ with $\thick$ and $\tower$ with $\tt$. The only real difference is that, instead of showing MA implies $\tt = \continuum$ directly, we will prove that $\tt = \tower$.

\begin{theorem}\label{thm:cardinals}
$\tt = \mathfrak{t}$.
\end{theorem}
\begin{proof}
We will prove that $\mathfrak{p} \leq \tt \leq \mathfrak{t}$. This is enough because, by a deep new result of Malliaris and Shelah, $\mathfrak{p} = \mathfrak{t}$ (see \cite{M&S}).

To prove $\tt \leq \mathfrak{t}$ we make use of the standard trick: supposing $\k < \tt$, we will show $\k < \mathfrak{t}$. Suppose $\k < \tt$ and let $T = \set{A_\a}{\a < \k}$ be a chain in $\Inf$. For each $n$, let $I_n = [2^n,2^{n+1}-1]$. For each $\a < \k$, let $B_\a = \bigcup_{n \in A}I_n$. Clearly each $B_\a$ is thick and $\set{B_\a}{\a < \k}$ is a chain in $\thick$. Since $\k < \tt$, there is some $B_\k$ such that $B_\k \sub^* B_\a$ for all $\a < \k$. Setting $A_\k = \set{n}{B_\k \cap I_n \neq \0}$, it is easily checked that $A_\k \sub^* A_\a$ for all $\a < \k$. Since $T$ was arbitrary, $\k < \mathfrak{t}$.

We use the same trick to prove $\pin \leq \tt$. Suppose $\k < \mathfrak{p}$ and let $\set{T_\a}{\a < \k}$ be a chain in $\thick$; we will show that there is some $T_\k \in \thick$ such that $T_\k \sub^* T_\a$ for every $\a$.

Let $\mathcal T = \set{T_\a - n}{\a < \k, n \in \N}$. We claim that $\mathcal T$ is a filter on $\Inf$. Fix $\a_0,\dots,\a_m$ and $n_0,\dots,n_m$. There is a thick set $T \sub \bigcap_{k \leq m}T_{\a_k}$; in fact, some cofinite subset of $T_{\max \{\a_0,\dots,\a_n\}}$ will do. Let $N = \max \{n_0,\dots,n_m\}$. Because $T$ is thick, there is an infinite $T' \sub T$ such that for every $k \in T'$ we have $[k,k+N] \sub T$. But if $[k,k+N] \sub T$ then $k \in \bigcap_{k \leq m}(T_{\a_k} - n_k)$; thus $\bigcap_{k \leq m}(T_{\a_k} - n_k)$ is infinite.

Hence $\mathcal T$ is a filter on $\Inf$. Since $\card{\mathcal T} = \k < \mathfrak{p}$, $\mathcal T$ has a lower bound $A$ in $\Inf$. Let $\set{a_n}{n \in \N}$ be an increasing enumeration of $A$.

For each $\a < \k$ and $m \in \N$, there is some minimal $\ell \in \N$ such that
$$\set{a_n}{n \geq \ell} \sub \bigcap_{k \leq m}(T_\a - k).$$
Let $f_\a(m) = \ell$. Equivalently, $f_\a(m)$ is the least element of $\N$ such that if $n \geq f_\a(m)$ then $[a_n,a_n+m] \sub T_\a$. It is well-known that $\mathfrak{p} \leq \mathfrak{b}$ (see, e.g., \cite{EvD}). Therefore $\k < \mathfrak{b}$ and there is some $f: \N \to \N$ such that, for all $\a < \k$, $f_\a <^* f$. Let
$$T_\k = \bigcup_{n \in \N}[a_{f(n)},a_{f(n)}+n].$$
Clearly $T_\k \in \thick$. Fix $\a < \k$. There is some $N \in \N$ such that $f_\a(n) < f(n)$ for $n \geq N$. But if $f_\a(n) < f(n)$ then $[a_{f(n)},a_{f(n)}+n] \sub T_\a$. Thus $T_\k \sub^* T_\a$. Hence the chain $\set{T_\a}{\a < \k}$ is not unbounded in $\thick$, which shows $\k < \tt$.
\end{proof}

We leave it an open question whether there is an ``elementary'' proof that $\tt = \tower$, i.e., one that avoids the power of Malliaris and Shelah's equality. We also leave as an exercise the (easy) proof that MA can be used directly to show $\tt = \continuum$.

\begin{theorem}\label{thm:main}
If $\tower = \continuum$ then there is a minimal right ideal in $\N^*$ that is also a $P$-set.
\end{theorem}
\begin{proof}
Using Lemma~\ref{lem:thickultrafilters}, it suffices to construct a $\thick$-ultrafilter that is also a $P$-filter.

Fix an enumeration $\set{A_\a}{\a < \continuum}$ of $\thick$. We construct a (reverse well ordered) chain $T = \set{X_\a}{\a < \continuum}$ in $\thick$ as follows. Set $X_0 = \N$. If $X_\a$ has already been defined, let $X_{\a+1} = X_\a \cap A_\a$ if $X_\a \cap A_\a \in \thick$, and otherwise let $X_{\a+1} = X_\a$. For limit $\a$, let $X_\a$ be any lower bound in $(\thick,\sub^*)$ of $\set{X_\b}{\b < \a}$; a lower bound exists because $\a < \tt$.

The chain $T$ constructed in the previous paragraph is a filter base; let $\F$ be the filter generated by $T$. If $A_\a \in \thick$, then either $A_\a \sub X_{\a+1}$, which implies $A_\a \in \F$, or $A_\a \cap X_\a$ is not thick. Thus $\F$ is a $\thick$-ultrafilter. If $\mathcal A$ is a countable subset of $\F$ then, shrinking the elements of $\mathcal A$ if necessary, we may assume $\mathcal A \sub T$ and write $\mathcal A = \set{X_{\a_n}}{n \in \w}$. By K\"{o}nig's Lemma, there is some $\a < \continuum$ with $\a_n < \a$ for all $n$. Then $X_\a$ is a pseudo-intersection for $\mathcal A$ in $\F$, so that $\F$ is a $P$-filter. 
\end{proof}

\begin{remark}
The minimal right ideal constructed in the previous proof is actually a $P_{\mathrm{cf}(\continuum)}$-set. Thus, for example, MA implies that some minimal right ideal is a $P_\continuum$-set.
\end{remark}

%\begin{proposition}\label{prop:MA}
%Martin's Axiom implies $\tt = \continuum$.
%\end{proposition}
%\begin{proof}
%Assume MA$(\k)$: we will show $\k < \tt$. Let $T = \set{X_\a}{\a < \k} \sub \thick$ satisfy $X_\a \sub^* X_\b$ whenever $\b < \a$. $T$ is clearly a filter base: let $\F$ be the $\thick$-filter generated by $T$. Let
%$$\P = \set{(s,A)}{s \in [\w]^{<\w}, A \in \F, \max s < \min A},$$
%with $(t,B) \leq (s,A)$ if and only if $s \sub t$, $A \supseteq B$, and $t \setminus s \sub A$. Variations of this poset are well-known (see Theorem 4.5 in \cite{Bth} or Lemma III.3.22 of \cite{Kun}), and it is easy to prove that $\P$ is ccc (in fact $\s$-centered).

%For each $\a < \k$, define
%$$D_\a = \set{(s,A) \in \P}{A \sub X_\a}$$
%and for each $n \in \N$ define
%$$E_n = \set{(s,A) \in \P}{s \text{ contains an interval of length } n}.$$
%It is easy to check that each of these sets is dense in $\P$.

%By MA$(\k)$, there is some $G \sub \P$ that meets each of the $D_\a$ and each of the $E_n$. Set $G = \bigcup \set{s}{(s,A) \in G}$. Because $G$ meets each of the $E_n$, we have $G \in \thick$. Because $G$ meets each of the $D_\a$, we have $G \sub^* X_\a$ for each $\a < \k$ (notice that $(s,A)$ forces $G \setminus s \sub A$). Thus $G$ is a lower bound for $T$ in $\thick$. As $T$ was arbitrary, this shows $\k < \tt$.
%\end{proof}

As remarked in the final section of \cite{Bls}, the elements of minimal right ideals, specifically the idempotent ones, are never $P$-points. In fact, if $I$ is a minimal right ideal and $p \in I$, then $\set{\s^n(p)}{n \in \N}$ is dense in $I$ by Proposition~\ref{prop:ideals}; since the minimal right ideals are separable, no point of one is even a weak $P$-point. Our result states that (sometimes, in some models) this is the only reason a minimal ultrafilter fails to be a $P$-point.

\section{More on $P$-set right ideals}\label{sec:independence}

We begin this section by proving that the conclusion of Theorem~\ref{thm:main} is consistently false.

\begin{lemma}\label{lem:ptop}
If $f: \N \to \N$ is finite-to-one, then $\b f$ maps closed $P$-sets to closed $P$-sets.
\end{lemma}
\begin{proof}
This lemma is a well-known bit of folklore, but we will give the short proof here for completeness.

Let $\F$ be a filter with $\F^*$ a closed $P$-set, and let $\G$ denote the unique filter such that $\b f(\F^*) = \G^*$. One can easily check that
$$\G = \set{A \sub \N}{f^{-1}(A) \in \F}.$$
Let $\mathcal A \sub \G$ be countable. Since $\F$ is a $P$-filter, $\F$ contains a pseudo-intersection $B$ for $\set{f^{-1}(A)}{A \in \mathcal A}$. $\b f(B) = \set{n}{A_0 \cap f^{-1}(n) \neq \0}$ is a pseudo-intersection for $\mathcal A$, so $\G$ is a $P$-filter as well.
\end{proof}

\begin{lemma}\label{lem:finitetoone}
There is a finite-to-one map $f: \N \to \N$ such that if $I$ is a minimal right ideal then $\b f(I)$ is a single point. Furthermore, for every $p \in \N^*$ there is some minimal right ideal $I$ such that $\b f(I) = p$.
\end{lemma}
\begin{proof}
Let $\set{a_n}{n \in \N}$ be an increasing sequence of natural numbers such that $\lim_{n \to \infty}(a_{n+1}-a_n) = \infty$, and let $f$ be the map that sends $m$ to the greatest $n$ such that $a_n \leq m$. We will show that $f$ has the required property.

For each $n$, let $I_n = f^{-1}(n)$. Let $\F$ be a $\thick$-ultrafilter. Letting
$$\b f (\F) = \set{A \sub \N}{f^{-1}(A) \in \F}$$
as in the proof of Lemma~\ref{lem:ptop}, it suffices to check that $\b f(\F)$ is an ultrafilter on $\N$. We will show that for every $A \sub \N$ either $f^{-1}(A) \in \F$ or $f^{-1}(\N \setminus A) \in \F$.

Suppose this is not the case, and let $A$ be such that neither $f^{-1}(A)$ nor $f^{-1}(\N \setminus A)$ is in $\F$. Since $\F$ is a $\thick$-ultrafilter, this means that there are $B,C \in \F$ such that neither $f^{-1}(A) \cap B$ nor $f^{-1}(\N \setminus A) \cap C$ is thick. $\F$ is a filter, so $B \cap C \in \F$, and neither $D = B \cap C \cap f^{-1}(A)$ nor $E = B \cap C \cap f^{-1}(\N \setminus A)$ is thick. Fix $k$ such that neither $D$ nor $E$ contains an interval of length $\geq k$. There is some $M$ such that if $n > M$ then $a_{n+1} - a_n > k$. If $I$ is an interval in $D \cup E$ with $\min I \geq a_M$, $I$ has length at most $2k$ (otherwise either $D$ or $E$ contains an interval of length greater than $k$). Thus $D \cup E$ is not thick. Since
$$D \cup E = (B \cap C \cap f^{-1}(A)) \cup (B \cap C \cap f^{-1}(\N \setminus A)) = B \cap C$$
and $B \cap C \in \F$, this is a contradiction.

For the second claim, observe that for any $p \in \N^*$,
$$\set{\bigcup_{n \in A}[a_n,a_{n+1}-1]}{A \in p}$$
is a thick filter. If $\F$ is a $\thick$-ultrafilter extending this filter, then $\F^*$ is the required minimal right ideal.
\end{proof}

We remark that Blass and Hindman in \cite{B&H} have proved a similar result. They show that if MA holds, then the (finite-to-one) map $f$ on $\N$ that sends $n$ to the leftmost $1$ in its binary expansion has the property that $\b f$ sends an idempotent ultrafilter to a Ramsey ultrafilter. Our lemma shows that if $p$ is any ultrafilter, then this map (or one of many others like it) will send some minimal idempotent to $p$.

\begin{theorem}\label{thm:noppoints}
It is consistent with \emph{ZFC} that no minimal right ideal is a $P$-set.
\end{theorem}
\begin{proof}
By Lemmas \ref{lem:ptop} and \ref{lem:finitetoone}, if some minimal right ideal is a $P$-set then there is a $P$-point in $\N^*$. By a famous result of Shelah, it is consistent with ZFC that $\N^*$ has no $P$-points (see \cite{Wim} or \cite{Sh2}).
\end{proof}

The proof of Theorem \ref{thm:noppoints} suggests the following problem:

\begin{question}
Is it consistent that $\N^*$ has a $P$-point but that no minimal right ideal is a $P$-set?
\end{question}

Recall that $p$ is \textbf{idempotent} if and only if $p+p = p$. The idempotents of $\N^*$ admit a natural partial order as follows: $p$ and $q$ are idempotent ultrafilters, then $p \leq q$ if and only if $p+q = q+p = p$. It is known (see, e.g., \cite{Brg}, Theorem 2.7) that an idempotent $p$ is minimal with respect to this order if and only if $p$ belongs to some minimal right ideal. Such ultrafilters are called \textbf{minimal idempotents}. It is also known (\cite{Brg} again) that if $q$ is any idempotent then there is a minimal idempotent $p$ with $p \leq q$.

We now give three corollaries of our Theorem~\ref{thm:main}:

\begin{corollary}\label{cor:twoquestions}
Suppose $\tower = \continuum$. Then 
\begin{enumerate}
\item the minimal right ideals are not homeomorphically embedded in $\N^*$.
\item there is a minimal right ideal $I$ such that, for any $p,q \in \b\N$, $p+q \in I$ if and only if $p \in I$.
\item there is a (necessarily minimal) idempotent $p$ that does not compare with any other idempotents under $\leq$.
\end{enumerate}
\end{corollary}
\begin{proof}
$(1)$ Let $K = \bigcup \set{I \sub \N^*}{I \text{ is a minimal right ideal}}$ and fix $p \in \N^* \setminus K$. (Note: by Corollary 4.41 in \cite{H&S}, $q \in \closure K$ if and only if every member of $q$ is piecewise syndetic, from which it easily follows that $K$ is nowhere dense and some such $p$ exists). The orbit closure of $p$, namely $\closure{\set{\s^n(p)}{n \in \N}}$, is subsystem of $\b\N$. By Zorn's Lemma, every subsystem of $\b\N$ contains a minimal subsystem $I$; in particular, there is a minimal right ideal $I \sub \closure{\set{\s^n(p)}{n \in \N}}$. Since $I$ is closed under $\s^{-1}$ and $p \notin I$, $\s^n(p) \notin I$ for all $n$. It follows that $I$ is not a $P$-set.

$(2)$ Let $I$ be a $P$-set and a minimal right ideal. If $p \in I$, then $p+q \in I$ because $I$ is a right ideal. If $p \notin I$, then, $p + q = \qlim \s^n(p)$ is an element of $\closure{\set{\s^n(p)}{n \in \N}}$. Since, as in $(1)$, $\s^n(p) \notin I$ for all $n$, we have $I \cap \closure{\set{\s^n(p)}{n \in \N}} = \0$. In particular, $p + q \notin I$.

$(3)$ Let $I$ be a minimal right ideal satisfying $(2)$, and let $p \in I$ be idempotent. Let $q$ be any idempotent other than $p$. If $q \in I$ then $q$ is $\leq$-minimal, so $p \not \leq q$. If $q \notin I$ then $q + p \notin I$ by $(2)$, in which case $p \not \leq q$.
\end{proof}

We remark that our work here is not necessary for proving the consistency of $(1)$: there are $2^\continuum$ minimal right ideals and, in a model due to Shelah, there are only $\continuum$ autohomeomorphisms of $\N^*$. That this conclusion follows from $\tower = \continuum$ is a new result, and we do not know whether it follows from ZFC.

\begin{question}
Is any of the three propositions listed in Corollary~\ref{cor:twoquestions} provable from \emph{ZFC}?
\end{question}

It may be possible to solve all three parts of this question at once. Recall that a \textbf{weak $P$-set} is a closed $X \sub \N^*$ such that if $Y \sub \N^* \setminus X$ is countable then $X \cap \closure{Y} = \0$. The proof of Corollary~\ref{cor:twoquestions} does not require that we have a minimal right ideal be a $P$-set: having a minimal right ideal that is a weak $P$-set is enough. In ZFC, it is impossible to prove that $P$-points exist (by the aforementioned result of Shelah), but it is possible to prove that weak $P$-points exist (a result of Kunen; see \cite{Kun}). Perhaps a similar situation holds for the minimal right ideals.

\begin{question}
Is some minimal right ideal a weak $P$-set?
\end{question}

While ZFC is not strong enough to prove that some minimal right ideal is a $P$-set, we can show that some ``small'' right ideal is a $P$-set.

\begin{theorem}\label{thm:Pset}
There is a nonempty, nowhere dense $P$-set $X \sub \N^*$ such that $\s(X) = X$. In particular, there is a nowhere dense $P$-set that is also a closed right ideal.
\end{theorem}
\begin{proof}
The second claim of the theorem follows from the first claim and Proposition~\ref{prop:ideals}.

By Lemma~\ref{lem:thickfilters}, it suffices to find a $\thick$-filter $\F$ such that $\F^*$ is nowhere dense, every countable subset of $\F$ has a pseudo-intersection in $\F$, and, for every $A \in \F$, $A+1, A-1 \in \F$.

Let $B = \set{b_n}{n \in \N}$ satisfy $\lim_{n \to \infty}(b_{n+1}-b_n) = \infty$. Let $\G$ be a filter such that $\G^*$ is a nowhere dense $P$-set (it is easy to prove that such a $\G$ exists: take a maximal chain in $\Inf$ and let $\G$ be the filter generated by it). Let $\B$ denote the family of all functions $f: \N \to \N$ such that $\lim_{n \to \infty}f(n) = \infty$. By a standard diagonalization argument, if $f_0, f_1, f_2, \dots \in \B$ then there is some $f \in \B$ such that $f <^* f_n$ for every $n$.

For every $G \in \G$ and $f \in \B$, let
$$X(G,f) = \bigcup_{n \in G} [b_n-f(n),b_n+f(n)].$$
Let $G_0, G_1 \in \G$ and $f_0,f_1 \in \B$. $G_2 = G_0 \cap G_1 \in \G$, and if we define $f_2: \N \to \N$ by $f_2(n) = \min\{f_0(n),f_1(n)\}$ then $f_2 \in \B$. Clearly $X(G_2, f_2) \sub X(G_0, f_0) \cap X(G_1, f_1)$. Thus $\set{X(G,f)}{G \in \G, f \in \B}$ is a filter base; let $\F$ denote the filter generated by it. $\F$ is clearly a $\thick$-filter; we now show that it has the other required properties.

For $f \in \B$, let $f^-(n) = f(n)-1$ and note that $f^- \in \B$ whenever $f \in \B$. If $G \in \G$ and $f \in \B$, it follows from our definitions that $X(G,f)+1 \supseteq X(G,f^-)$ and $X(G,f)-1 \supseteq X(G,f^-)$. Consequently, for every $A \in \F$ we have $A+1, A-1 \in \F$.

For each $n$, let $G_n \in \G$ and $f_n \in \B$. Since $\G^*$ is a $P$-set, there is some $G_\infty \in \G$ that is a pseudo-intersection for the $G_n$. As noted above, there is also some $f_\infty \in \B$ such that $f_\infty <^* f_n$ for every $n$. Then $X(G_\infty,f_\infty)$ is a pseudo-intersection for $\set{X(G_n,f_n)}{n \in \N}$. Since the $X(G,f)$ form a basis for $\F$, $\F$ is a $P$-filter.

Finally we show that $\F^*$ is nowhere dense. Using Stone duality, we have that $\F^*$ is nowhere dense if and only if $\F$ is unbounded in $\Inf$. Suppose this is not the case, and let $A$ be a bound for $\F$.

For each $n$, let
$$\a(n) = \max \set{\min\{a-b_n,b_{n+1}-a\}}{a \in A \cap [b_n,b_{n+1}]}.$$
If the range of $\a$ is unbounded, then there is some function $f \in \B$ such that $\a \not <^* f$. In this case, $X(\N,f)$ misses infinitely many elements of $A$, contradicting the assumption that $A$ is a pseudo-intersection for $\F$.

Therefore $\a$ is bounded: there is some $k$ such that $\a(n) < k$ for all $n$. Fix $N$ large enough so that $b_{n+1} - b_n > 2k$ whenever $n > N$. Fix $f \in \B$ so that $k \leq f(n) \leq \lfloor \frac{b_n - b_{n-1}}{2} \rfloor$ for all $n > N$. Define $B = \set{n > N}{|b_n-a| < k \text{ for some }a \in A}$. Since $\G^*$ is nowhere dense, $\G$ has no pseudo-intersection; in particular, $B$ is not a pseudo-intersection for $\G$. Thus there is some $G \in \G$ such that $B \setminus G$ is infinite. Then $X(G,f)$ misses infinitely many elements of $A$ (namely all of those in $[b_n-k,b_n+k]$ for $n \in B \setminus G$), contradicting the assumption that $A$ is a pseudo-intersection for $\F$.
\end{proof}

\begin{corollary}
There is a nowhere dense closed $I \sub \N^*$ such that, for any $p,q \in \b\N$, $p+q \in I$ if and only if $p \in I$.
\end{corollary}
\begin{proof}
See the proof of Corollary~\ref{cor:twoquestions}$(3)$.
\end{proof}

\section{A consistent characterization of $\s$}\label{sec:shiftmap}

In this section we show that it is both consistent with and independent of ZFC that the shift map $\s$ is, up to isomorphism, the unique chain transitive autohomeomorphism of $\N^*$. The ``independence'' part of this result will use the nowhere dense $P$-set of Theorem~\ref{thm:Pset}.

Recall that two dynamical systems $(X,f)$ and $(Y,g)$ are \textbf{isomorphic} (or, for some authors, conjugate), written $(X,f) \iso (Y,g)$ or $X \iso Y$, if there is a homeomorphism $h: X \to Y$ such that the following diagram commutes:
\begin{center}
\begin{tikzpicture}[style=thick, xscale=.5,yscale=.5]

\draw (0,0) node {$X$} -- (0,0);
\draw (4,0) node {$X$} -- (4,0);
\draw (0,-4) node {$Y$} -- (0,-4);
\draw (4,-4) node {$Y$} -- (4,-4);
\draw (4.6,-2) node {$h$} -- (4.6,-2);
\draw (-.6,-2) node {$h$} -- (-.6,-2);
\draw (2,.5) node {$f$} -- (2,.5);
\draw (2,-3.5) node {$g$} -- (2,-3.5);

\draw[->] (4,-.5) -- (4,-3.5);
\draw[->] (0,-.5) -- (0,-3.5);
\draw[->] (.6,0) -- (3.4,0);
\draw[->] (.5,-4) -- (3.5,-4);

\end{tikzpicture}\end{center}
$Y$ is a \textbf{quotient} of $X$ (or is semi-conjugate to $X$) if $h: X \to Y$ is a continuous surjection, but not necessarily a homeomorphism, for which the above diagram commutes. In this case, $h$ is called a \textbf{quotient map} (or a semi-conjugation).

Chain transitivity is a well studied property of metric dynamical systems. While it is typically defined in terms of some fixed metric, it is actually a topological property (i.e., it is preserved by isomorphisms), and, replacing metrics with uniformities, chain transitivity can be defined for non-metrizable systems as well.

Explicitly, if $X$ is a dynamical system and $\U$ is an open cover of $X$, we say that $\seq{x_i}{i \leq n}$ if a $\U$\textbf{-chain} from $x_0$ to $x_n$ if, for every $i < n$, there is some $U \in \U$ such that $f(x_i) \in U$ and $x_{i+1} \in U$. Let us say that a (not necessarily metric) dynamical system $X$ is \textbf{chain transitive} if for every open cover $\U$ of $X$ and every $x,y \in X$, there is a $\U$-chain from $x$ to $y$. The reader can easily check that this is equivalent to the usual definition of chain transitivity for metric systems (use the Lebesgue Covering Lemma for the nontrivial direction).

The following lemma is essentially proved in \cite{Akn} as Theorem 4.12, although the notation is different and the proof there is for metric systems. For completeness, we give a proof here as well.

\begin{lemma}\label{lem:breakingtransitivity}
$(X,f)$ fails to be chain transitive if and only if there is some nonempty open $U \neq X$ and closed $C \sub U$ such that $f(U) \sub C$.
\end{lemma}
\begin{proof}
Suppose there is some open $U \neq X$ and closed $C \sub U$ such that $f(U) \sub C$. Then $\U = \{U,X \setminus C\}$ is an open cover of $X$, and there is no $\U$-chain from a point of $U$ to a point of $X \setminus U$.

For the other direction, suppose $X$ is not chain transitive. Fix some open cover $\U$ of $X$, and some $x,y \in X$ such that there no $\U$-chain from $x$ to $y$. Notice that if there is a $\U$-chain from $x$ to any point of some $V \in \U$ then there is a $\U$-chain from $x$ to every point of $V$. Let $\U_0$ be the set of all $V \in \U$ such that there is a $\U$-chain from $x$ to some (any) point of $V$, and let $\U_1 = \U \setminus \U_0$. Set $U = \bigcup \U_0$ and $C = X \setminus \bigcup \U_1$. $\0 \neq U \neq X$ because $x \in U \not\ni y$. Because $\U$ covers $X$, we have $C \sub U$. If $f(U) \setminus C \neq \0$, then clearly it is possible to find a $\U$-chain from $x$ to a point of $X \setminus C = \bigcup \U_1$, contradicting the definition of $\U_1$. Thus $f(U) \sub C$.
\end{proof}

\begin{lemma}
A map $F: \N^* \to \N^*$ is chain transitive if and only if, for every infinite, co-infinite $A \sub \N$, $F(A^*) \not\sub A^*$.
\end{lemma}
\begin{proof}
Let $U$ be an open subset of $\N^*$ with $U \neq \N*$, and let $C \sub U$ be nonempty and closed. Since $C$ is closed, $C = \bigcap \set{A^*}{A^* \supseteq C}$. By Corollary 3.1.5 in \cite{Eng}, there is some finite set $\set{A_i}{i \leq n}$ such that $\bigcap_{i \leq n}A_i^* \sub U$. Since $\bigcap_{i \leq n}A_i^* = (\bigcap_{i \leq n}A_i)^*$, there is some $A \sub \N$ such that $C \sub A^* \sub U$. If $f(U) \sub C$ then $f(A^*) \sub A^*$, so the result follows from Lemma~\ref{lem:breakingtransitivity}.
\end{proof}

\begin{theorem}\label{thm:N*chaintransitive}
The shift map on $\N^*$ is chain transitive.
\end{theorem}
\begin{proof}
Suppose $A$ is infinite and co-infinite. Then $(A+1) \cap (\N \setminus A) = \set{n \in \N}{n \in A \text{ and } n+1 \notin A}$ is infinite. In particular,
$$\s(A^*) \cap (\N^* \setminus A^*) = (A+1)^* \cap (\N \setminus A)^* = ((A+1) \cap (\N \setminus A))^* \neq \0.$$
This shows that $\s(A^*) \not\sub A^*$.
\end{proof}

\begin{lemma}\label{lem:chaintransitivequotients}
If $X$ is chain transitive and $Y$ is a quotient of $X$, then $Y$ is chain transitive.
\end{lemma}
Again, this is proved in \cite{Akn} for metric spaces, and the proof generalizes easily to the non-metric case.
\begin{proof}
Let $Q: X \to Y$ be a quotient map, let $\U$ be an open cover of $Y$, and let $y,z \in Y$. $\mathcal V = \set{Q^{-1}(U)}{U \in \U}$ is an open cover of $X$. Fixing some $x_y \in Q^{-1}(y)$ and $x_z \in Q^{-1}(z)$, there is a $\mathcal V$-chain $\seq{x_i}{i \leq n}$ in $X$ from $x_y$ to $x_z$. Then $\seq{Q(x_i)}{i \leq n}$ is a $\U$-chain from $y$ to $z$.
\end{proof}

We say that a map $F: \N^* \to \N^*$ is \textbf{trivial} if there is some map $f: \N \to \N$ such that $F = \b f \rest \N^*$.

\begin{lemma}\label{lem:trivialmaps}
If $\t$ is any chain transitive trivial map on $\N^*$, then $(\N^*,\t) \iso (\N^*,\s)$.
\end{lemma}
\begin{proof}
Suppose $t: \N \to \N$ is a function such that $\t = \b t \rest \N^*$ is a chain transitive self-map of $\N^*$. By continuity, $\t(A^*) = (t(A))^*$ for every $A \sub \N$. If $A \sub \N$ and $t(A) \sub^* A$, then $\t(A^*) = A^*$. If $\0 \neq A^* \neq \N^*$, this contradicts the chain transitivity of $\t$ by Lemma~\ref{lem:breakingtransitivity}. Thus, unless $A$ is either finite or cofinite, $t(A) \setminus A$ must be infinite.

For each $n$, let $\mathcal O(n) = \set{t^m(n)}{n \in \N}$. If $\mathcal O(n)$ is finite, then $n$ is periodic under $t$ and $\mathcal O(n)$ is invariant under $t$. Suppose there are infinitely many $n$ with $\mathcal O(n)$ finite. If $\mathcal A$ is an infinite, co-infinite subset of $\set{\mathcal O(n)}{\card{\mathcal O(n)} < \aleph_0}$, then $\bigcup \mathcal A$ is infinite and co-infinite, and is also invariant under $t$. This contradicts the conclusion of the previous paragraph, so there are finitely many $n$ with $\mathcal O(n)$ finite.

In particular, $\mathcal O(n)$ is infinite for some $n$. Since $\mathcal O(t(n)) \setminus \mathcal O(n) = \{n\}$, we have $t(\mathcal O(n)) = \mathcal O(t(n)) \sub^* \mathcal O(n)$. This implies $\mathcal O(n)$ is cofinite, again using the first paragraph of our proof.

Fix $n$ with $\mathcal O(n)$ cofinite and let $\varphi: \N \to \N$ be given by $\varphi(m) = t^m(n)$. Consider $\b \varphi: \b\N \to \b\N$. Since $\varphi$ is a bijection of $\N$ onto a cofinite subset of $\N$, $\b \varphi$ restricts to an autohomeomorphism $\varphi^*: \N^* \to \N^*$. For every $A \sub \N$, we have $\varphi(A+1) = t(\varphi(A))$. By continuity, $\varphi^*(\s(A^*)) = \varphi^*((A+1)^*) = \t(\varphi^*(A))$ for every $A$. It follows that, for every $p \in \N^*$, $\varphi^*(\s(p)) = \t(\varphi^*(p))$. Thus $\varphi^*$ is an isomorphism $(\N^*,\s) \to (\N^*,\t)$.
\end{proof}

Borrowing a famous result of Shelah, we now have a consistent characterization of $\s$:

\begin{theorem}\label{thm:consistency}
It is consistent with \emph{ZFC} that (up to isomorphism) $\s$ is the only chain transitive autohomeomorphism of $\N^*$.
\end{theorem}
\begin{proof}
Shelah proves in \cite{Shl} that it is consistent with ZFC that every autohomeomorphism of $\N^*$ is trivial.
\end{proof}

The next theorem gives the complementary result:

\begin{theorem}
Assuming the Continuum Hypothesis, there is a chain transitive autohomeomorphism of $\N^*$ that is not isomorphic to $\s$.
\end{theorem}
\begin{proof}
Using Theorem~\ref{thm:Pset} (or Theorem~\ref{thm:main} plus CH), let $X \sub \N^*$ be a nowhere dense $P$-set with $\s(X) = X$. Let $Y$ be the quotient space obtained from $\N^*$ from the relation
$$p \sim q \qquad \Leftrightarrow \qquad p = q \ \text{ or } \ p,q \in X.$$
That is, $Y$ is the compact Hausdorff space obtained from $\N^*$ by collapsing $X$ to a point.

Because $\s(X) = X$, $[\s(x)] = [\s(y)]$ whenever $x \sim y$; define $\t$ on $Y$ by setting $\t([x]) = [\s(x)]$. It is easy to check that this map is continuous, so that $(Y,\t)$ is a dynamical system. Clearly $\t$ is a bijection on $Y$ (because $\s$ is a bijection on $\N^*$), so that $\t$ is a continuous bijection, hence an autohomeomorphism, on $Y$. By construction $(Y,\t)$ is a quotient of $(\N^*,\s)$ and is therefore chain transitive by Lemma~\ref{lem:chaintransitivequotients}.

Because $X$ is nowhere dense in $\N^*$, Corollary 1.2.4 in \cite{JvM} implies that $Y \homeo \N^*$. Thus $\t$ is a chain transitive autohomeomorphism of $\N^*$. We note that Corollary 1.2.4 in \cite{JvM} requires CH; in fact, it has been shown that it is equivalent to CH (see \cite{vDM}).

To complete the proof, we observe that $\t([x]) = [x]$ for $x \in X$. Since $\s$ has no fixed points, $(\N^*,\s)$ and $(Y,\t)$ cannot be isomorphic.
\end{proof}

Roughly speaking, there is a lot of distance between CH and Shelah's model where every autohomeomorphism of $\N^*$ is trivial. This leaves open the question of whether there is more than one chain transitive autohomeomorphism of $\N^*$ under other natural hypotheses, for example MA. We also do not know whether another adjective or two can be added to ``chain transitive autohomeomorphism'' to obtain a topological characterization of $\s$ in ZFC.

%\section{Other stuff}

%\begin{theorem}
%Suppose a minimal right ideal $I$ is not a weak $P$-set. Then there is a countable set $\mathcal I$ of minimal right ideals such that $I \sub \closure{\bigcup \mathcal I}$.
%\end{theorem}
%\begin{proof}
%Fix a countable $D \sub \N^*$ such that $\closure{D} \cap I \neq \0$. Fix any minimal right ideal $M$ and, for each $d \in D$, let $I_d = d + M$. By Lemma????? in \cite{H&S}, $I_d$ is a minimal right ideal. Set $\mathcal I = \set{I_d}{d \in D}$.

%Let $i \in I \cap \closure{D}$ and $m \in M$. By continuity, $i+m \in \closure{\set{d+m}{d \in D}}$. Since $m$ was arbitrary, we have $i+m \in \closure{\bigcup \mathcal I}$ for every $m \in M$. Thus $i+M \sub \closure{\bigcup \mathcal I}$. By Lemma????? in \cite{H&S}, $i+M = I$.
%\end{proof}

%Let $K$ denote the union of all minimal right ideals in $\N^*$. $K$ is the smallest two-sided ideal in $\N^*$ (see \cite{H&S}, Theorem ?????). In particular, $K$ is a subsemigroup of $\N^*$.

%\begin{theorem}$ \ $
%Let $I$ be a minimal right ideal. The following are equivalent:
%\begin{enumerate}
%\item $I$ is a weak $P$-set.
%\item For any $p,q \in \N^*$, $p+q \in I$ if and only if $p \in I$.
%\end{enumerate}
%\end{theorem}
%\begin{proof}
%Suppose $I$ is not a weak $P$-set and fix a countable $D \sub \N^*$ such that $\closure{D} \cap I \neq \0$. Fix $p \in \closure{D} \cap I$. Let $M \neq I$ be a minimal right ideal and let $m \in M$. By continuity, $p+m \in \closure{\set{d+m}{d \in D}}$.

%??????
%\end{proof}

\end{document}